\documentclass[12pt,leqno]{amsart}  

\usepackage{calc}
\usepackage[left=1in,right=1in,top=1in,bottom=1in]{geometry}  
\usepackage{graphicx}
\usepackage{amsmath,amssymb,epsfig}
\usepackage{amsthm}
\usepackage{enumerate}

\newcommand{\N}{{\mathbf N}}
\newcommand{\C}{{\mathbb C}}
\newcommand{\R}{{\mathbb R}}
\newcommand{\Z}{{\mathbb Z}}
\newcommand{\Q}{{\mathbb Q}}
\newcommand{\T}{{\mathcal T}}

\newcommand{\Gal}{{\mathrm{Gal}}}
\newcommand{\Sel}{{\mathrm{Sel}}}
\newcommand{\dimF}{{\mathrm{dim}_{\mathbb{F}_2}}}
\newcommand{\Ftwo}{{\mathbb{F}_2}}
\newcommand{\F}{{\mathbb{F}}}

\newcommand{\Zt}{{\mathbb{Z}/2\mathbb{Z}}}

\newcommand{\hatphi}{{ \hat \phi }}

\newcommand{\p}{{ \mathfrak{p} }}

\newcommand{\ord}{{ \mathrm{ord} }}

\newtheorem{theorem}{\bf{Theorem}}[section]

\newtheorem{proposition}[theorem]{\bf{Proposition}}

\newtheorem{lemma}[theorem]{\bf{Lemma}}

\theoremstyle{definition}

\newtheorem{remark}[theorem]{\bf{Remark}}

\input xy 
\xyoption{all}

\usepackage[OT2,T1]{fontenc}

\DeclareSymbolFont{cyrletters}{OT2}{wncyr}{m}{n}
\DeclareMathSymbol{\Sha}{\mathalpha}{cyrletters}{"58}

\begin{document}
\bibliographystyle{plain}


\title[{Selmer ranks of twists of elliptic curves with partial two-torsion}]{The distribution of 2-Selmer ranks of quadratic twists of elliptic curves with partial two-torsion}

\author{Zev Klagsbrun}
\email{zev.klagsbrun@gmail.com}

\author{Robert J. Lemke Oliver}
\email{rjlo@stanford.edu}



\begin{abstract}This paper presents a new result concerning the distribution of 2-Selmer ranks in the quadratic twist family of an elliptic curve over an arbitrary number field $K$ with a single point of order two that does not have a cyclic 4-isogeny defined over its two-division field. We prove that at least half of all the quadratic twists of such an elliptic curve have arbitrarily large 2-Selmer rank, showing that the distribution of 2-Selmer ranks in the quadratic twist family of such an elliptic curve differs from the distribution of 2-Selmer ranks in the quadratic twist family of an elliptic curve having either no rational two-torsion or full rational two-torsion.
\end{abstract}
\maketitle



\section{Introduction}

\subsection{Distributions of Selmer Ranks}

Let $E$ be an elliptic curve defined over a number field $K$ and let $\Sel_2(E/K)$ be its 2-Selmer group (see Section \ref{bg} for its definition). We define the \textbf{2-Selmer rank of $E/K$}, denoted $d_2(E/K)$, by \begin{equation*}d_2(E/K) =  \dimF  \Sel_2(E/K) - \dimF E(K)[2].\end{equation*}

In 1994, Heath-Brown proved that the 2-Selmer ranks of all of the quadratic twists of the congruent  number curve $E/\Q$ given by $y^2 = x^3 - x$ had a particularly nice distribution. In particular, he showed that there are explicits constants $\alpha_0, \alpha_1, \alpha_2, \ldots$ summing to one such that \begin{equation*}\lim_{X \rightarrow \infty}\frac{|\{ d  \text{ squarefree } |d| < X  : d_2(E^d/\Q) = r  \} |}{|\{ d  \text{ squarefree } |d| < X \}|} = \alpha_r\end{equation*} for every $r \in  \Z^{\ge 0}$, where $E^d$ is the quadratic twist of $E$ by $d$ \cite{HB}.  This result was extended by Swinnerton-Dyer and Kane to all elliptic curves $E$ over $\Q$ with $E(\Q)[2] \simeq \Zt \times \Zt$ that do not have a cyclic 4-isogeny defined over $K$ \cite{Kane}, \cite{SD}.

A similar result was obtained by  Klagsbrun, Mazur, and Rubin for elliptic curves $E$ over a general number field $K$ with $\Gal(K(E[2])/K) \simeq \mathcal{S}_3$, where squarefree $d$ are replaced by quadratic characters of $K$ and a suitable ordering of all such characters is taken \cite{KMR2}. 

In this work we show that this type of result does not hold when $E(K)[2] \simeq \Zt$. In particular, we prove the following:

\begin{theorem}\label{mainthm} For $d \in \mathcal{O}_K$, let $\chi_d$ be the quadratic character of $K$ that cuts out the extension $K(\sqrt{d})$ and define $$C(K, X):=\{\chi_d: |\mathbf{N}_{K/\Q} d |< X\}.$$
Let $E$ be an elliptic curve defined over $K$ with $E(K)[2] \simeq \Z/2\Z$ that does not have a cyclic isogeny defined over $K(E[2])$. Then for any fixed $r$, \begin{equation*}\liminf_{X\rightarrow \infty}\frac{ \left | \{ \chi \in C(K, X) : d_2(E^\chi/K) \ge r \} \right | }{|C(K, X)|}  \ge \frac{1}{2}\end{equation*} where $E^\chi$ is the quadratic twist of $E$ by any $d \in \mathcal{O}_K$ with $\chi_d = \chi$.
\end{theorem}

In particular, this shows that there is not a distribution function on 2-Selmer ranks within the quadratic twist family of $E$.

Theorem \ref{mainthm} is an easy consequence of the following result. 

\begin{theorem}\label{Tconverge}  Let $E$ be an elliptic curve defined over $K$ with $E(K)[2] \simeq \Z/2\Z$ that does not have a cyclic isogeny defined over $K(E[2])$. Then the normalized distribution \begin{equation*}\displaystyle \frac{P_r(\T(E/E^\prime), X)}{\sqrt{\frac{1}{2} \log \log X}}\end{equation*} converges weakly to the Gaussian distribution \begin{equation*}G(z) = \frac{1}{\sqrt{2\pi}} \int_{-\infty}^{z} \! e^\frac{-w^2}{2} \, \mathrm{d}w,\end{equation*} where \begin{equation*} P_r(\T(E/E^\prime), X) = \displaystyle{ \frac{ |\{\chi \in C(K, X) : \ord_2 \T(E^\chi/E^{\prime \chi}) \le r \}| }{|C(K, X)|} }\end{equation*} for $X \in \R^+$, $r \in \Z^{\ge 0}$, and $\T(E^\chi/E^{\prime \chi})$ as defined in Section \ref{bg}.

\end{theorem}

In turn, Theorem \ref{Tconverge} follows from a variant of the Erd{\H o}s-Kac theorem for quadratic characters of number fields. Let $C(K)$ denote the set of all quadratic characters of $K$. For any $\chi \in C(K)$, we can associate a unique squarefree ideal $D_\chi$ to $\chi$ by taking $D_\chi$ to be the squarefree part of the ideal $\langle d \rangle$, where $d$ is any element of $\mathcal{O}_K$ such that $\chi_d = \chi$. We say that a function $f$ on the ideals of $\mathcal{O}_K$ is \textbf{additive} if $f(\mathfrak{a} \mathfrak{a}^\prime) = f(\mathfrak{a}) + f(\mathfrak{a}^\prime)$ whenever the ideals $\mathfrak{a}$ and $\mathfrak{a}^\prime$ are relatively prime; we define $f(\chi)$ for $\chi\in C(K)$ to be $f(D_\chi)$.




To an additive function $f$, we attach quantities $\mu_f(X)$ and $\sigma_f(X)$, defined to be
\[
\mu_f(X) := \sum_{\N\mathfrak{p} < X} \frac{f(\mathfrak{p})}{\N\mathfrak{p}}, \,\,\,\, \sigma_f(X):=\left(\sum_{\N\mathfrak{p} < X} \frac{f(\mathfrak{p})^2}{\N\mathfrak{p}}\right)^{1/2},
\]
and we prove the following.

\begin{theorem}\label{EKthm}
Suppose that $f$ is an additive function such that $0\leq f(\mathfrak{p})\leq 1$ for every prime $\mathfrak{p}$. If $\sigma_f(X)\to\infty$ as $X\to\infty$, then
\[
\lim_{X\to\infty} \frac{ | \left \{ \chi\in C(K, X) : f(\chi)-\mu_f(X) \leq z \cdot \sigma_f(X) \right \}|}{ |C(K, X)|} = G(z),
\]
\end{theorem}

In the special case where $K = \Q$, Theorem \ref{mainthm} follows from results recently obtained by Xiong about the distribution of $\dimF \Sel_\phi(E/\Q)$ in quadratic twist families \cite{X}. We are able to obtain results over general number fields by applying different methods to the coarser question about the distribution of $\ord_2 T(E/E^\prime)$ in twist families (see Theorem \ref{Tconverge}).  



\begin{remark}
The methods in the paper can easily be adapted to studying cubic twists of the $j$-invariant 0 curve $y^2 = x^3 + k$ for any number field $K$ with $K(\sqrt{k}, \sqrt{-3})/K$ biquadratic. In that case, Theorem \ref{mainthm} holds with 2-Selmer rank replaced by 3-Selmer rank and quadratic characters replaced by cubic characters.
\end{remark}

\subsection{Layout}
We begin in Section \ref{bg} by recalling the definitions of the 2-Selmer group and the Selmer groups associated with a 2-isogeny $\phi$ and presenting some of the connections between them. In Section \ref{twistplace}, we examine the behavior of the local conditions for the $\phi$-Selmer group under quadratic twist and show how that quantity $\T(E^\chi/E^{\prime \chi})$ can be related to the value $f(\chi)$ of an additive function $f$ on $C(K)$. Theorem \ref{EKthm} is proved in Section \ref{pfEKthm} and we conclude with the proofs of Theorems \ref{mainthm} and \ref{Tconverge} in Section \ref{pfofmain}.


\subsection*{Acknowledgement}
The first author would like to express his thanks to Karl Rubin for his helpful comments and suggestions, to Ken Kramer for a series of valuable discussions, and to Michael Rael and Josiah Sugarman for helpful conversations regarding the Erd{\H o}s-Kac theorem. The first author was supported by NSF grants DMS-0457481, DMS-0757807, and DMS-0838210.

\section{Selmer Groups}\label{bg}

We begin by recalling the definition of the 2-Selmer group. If $E$ is an elliptic curve defined over a field $K$, then $E(K)/2E(K)$ maps into $H^1(K, E[2])$ via the Kummer map. The following diagram commutes for every place $v$ of $K$, where $\delta$ is the Kummer map.

\begin{center}\leavevmode
\begin{xy} \xymatrix{
E(K)/2E(K)  \ar[d] \ar[r]^{\delta} & H^1(K, E[2]) \ar[d]^{Res_v} \\
E(K_v)/2E(K_v)   \ar[r]^{\delta} & H^1(K_v, E[2]) }
\end{xy}\end{center}

We define a distinguished local subgroup $H^1_f(K_v, E[2]) \subset H^1(K_v, E[2])$ as the image $\delta \left ( E(K_v)/2E(K_v) \right ) \subset H^1(K_v, E[2])$ for each place $v$ of $K$ and we define the \textbf{2-Selmer group} of $E/K$, denoted $\Sel_2(E/K)$, by \begin{equation*}\Sel_2(E/K) = \ker \left ( H^1(K, E[2]) \xrightarrow{\sum res_v} \bigoplus_{v\text{ of } K} H^1(K_v, E[2])/H^1_f(K_v, E[2]) \right ).\end{equation*}
The  $2$-Selmer group is a finite dimensional $\F_2$-vector space that sits inside the exact sequence of $\F_2$-vector spaces \begin{equation*}0\rightarrow E(K)/2E(K) \rightarrow \Sel_2(E/K) \rightarrow \Sha(E/K)[2] \rightarrow 0\end{equation*} where $\Sha(E/K)$ is the Tate-Shafaravich group of $E$.


If $E(K)$ has a single point of order two, then there is a two-isogeny $\phi:E \rightarrow E^\prime$ between $E$ and $E^\prime$ with kernel $C = E(K)[2]$. This isogeny gives rise to two Selmer groups.



We have a short exact sequence of $G_K$ modules \begin{equation} 0 \rightarrow C \rightarrow E(\overline{K}) \xrightarrow{\phi}  E^\prime(\overline{K}) \rightarrow 0\end{equation} which gives rise to a long exact sequence of cohomology groups \begin{equation*}0 \rightarrow C \rightarrow E(K) \xrightarrow{\phi} E^\prime(K) \xrightarrow{\delta} H^1(K, C) \rightarrow H^1(K, E) \rightarrow H^1(K, E^\prime) \ldots\end{equation*} The map $\delta$ is given by $\delta(Q)(\sigma)= \sigma(R) - R$ where $R$ is any point on $E(\overline{K})$ with $\phi(R) = Q$. 

This sequence remains exact when we replace $K$ by its completion $K_v$ at any place $v$, which gives rise to the following commutative diagram.

\begin{center}\leavevmode
\begin{xy} \xymatrix{
E^\prime(K)/\phi(E(K))  \ar[d] \ar[r]^\delta & H^1(K, C) \ar[d]^{Res_v} \\
E^\prime(K_v)/\phi(E(K_v))   \ar[r]^\delta & H^1(K_v, C) }
\end{xy}\end{center}

In a manner similar to how we defined the 2-Selmer group, we define distinguished local subgroups $H^1_\phi(K_v, C)\subset  H^1(K_v, C)$ as the image of $E^\prime(K_v)/\phi(E(K_v))$ under $\delta$ for each place $v$ of $K$.  We define the \textbf{$\mathbf \phi$-Selmer group of $\mathbf E$}, denoted $\Sel_\phi(E/K)$ as  \begin{equation*}\Sel_\phi(E/K) = \ker \left ( H^1(K, C) \xrightarrow{\sum res_v} \bigoplus_{v \text{ of } K} H^1(K_v, C)/H^1_\phi(K_v, C) \right ).\end{equation*} 

The isogeny $\phi$ on $E$ gives gives rise to a dual isogeny $\hat \phi$ on $E^\prime$ with kernel $C^\prime = \phi(E[2])$. Exchanging the roles of $(E, C, \phi)$ and $(E^\prime, C^\prime, \hat \phi)$ in the above defines the $\mathbf{\hat \phi}$\textbf{-Selmer group}, $\Sel_\hatphi(E^\prime/K)$, as a subgroup of $H^1(K, C^\prime)$. The groups $\Sel_\phi(E/K)$ and $\Sel_\hatphi(E^\prime/K)$ are finite dimensional $\Ftwo$-vector spaces and we can compare the sizes of the $\phi$-Selmer group, the $\hat \phi$-Selmer group, and the 2-Selmer group using the following two theorems.



\begin{theorem}\label{gss}The $\phi$-Selmer group, the $\hat \phi$-Selmer group, and the 2-Selmer group sit inside the exact sequence \begin{equation}0 \rightarrow E^\prime(K)[2]/\phi(E(K)[2]) \rightarrow \Sel_\phi(E/K) \rightarrow \Sel_2(E/K) \xrightarrow{\phi}\Sel_\hatphi(E^\prime/K).\end{equation}
\end{theorem}
\begin{proof}
This is a well known diagram chase. See Lemma 2 in \cite{FG} for example.
\end{proof}

The \textbf{Tamagawa ratio} $\T(E/E^\prime)$ defined as $\mathcal{T}(E/E^\prime) = \frac{ \big | \Sel_\phi(E/K)  \big |}{\big |\Sel_{\hat \phi}(E^\prime/K)\big |}$ gives a second relationship between the $\Ftwo$-dimensions of $\Sel_\phi(E/K)$ and $\Sel_\hatphi(E^\prime/K)$.


\begin{theorem}[Cassels]\label{prodform2}
The Tamagawa ratio $\mathcal{T}(E/E^\prime)$ is given by \begin{equation*}\mathcal{T}(E/E^\prime) = \prod_{v\text{ of } K}\frac{\left | H^1_\phi(K_v, C)\right |}{2}.\end{equation*}
\end{theorem}
\begin{proof}
This is a combination of Theorem 1.1 and equations (1.22) and (3.4) in \cite{Cassels8}.
\end{proof}

Stepping back, we observe that if $\T(E/E^\prime) \ge 2^{r+2}$ , then $d_\phi(E/K) \ge r+2$, and therefore by Theorem \ref{gss}, $d_2(E/K) \ge r$. (If $E$ does not have a cyclic 4-isogeny defined over $K$ then we can in fact show that $\T(E/E^\prime) \ge 2^r$ implies that $d_2(E/K) \ge r$, but this is entirely unnecessary for our purposes.)

\section{Local Conditions at Twisted Places}\label{twistplace}

For the remainder of this paper, we will let $E$ be an elliptic curve with $E(K)[2] \simeq \Zt$ and let $\phi:E\rightarrow E^\prime$ be the isogeny with kernel $C=E(K)[2]$.


If $\p \nmid 2$ is a prime where $E$ has good reduction, then $H^1_\phi(K_\p, C)$ is a 1-dimensional $\F_2$-subspace of $H^1(K_\p, C)$ equal to the unramified local subgroup $H^1_u(K_\p, C)$. If such a $\p$ is ramified in the extension $F/K$ cut out by a character $\chi$, then the twisted curve $E^\chi$ will have bad reduction at $\p$. The following lemma addresses the size of $H^1_\phi(K_\p, C^\chi)$.

\begin{lemma}\label{localconds}
Suppose $\p \nmid 2$ is a prime where $E$ has good reduction and $\p$ is ramified in the extension $F/K$ cut out by $\chi$.

\begin{enumerate}[(i)]
\item If  $E(K_\p)[2] \simeq \Zt \simeq E^\prime(K_\p)[2]$, then $\dimF H^1_\phi(K_\p, C^\chi) = 1$.
\item If  $E(K_\p)[2] \simeq \Zt \times \Zt \simeq E^\prime(K_\p)[2]$, then $\dimF H^1_\phi(K_\p, C^\chi) = 1$.
\item If $E(K_\p)[2] \simeq \Zt$ and $E^\prime(K_\p)[2] \simeq \Zt \times \Zt$, then $\dimF H^1_\phi(K_\p, C^\chi) = 2$.
\item If $E(K_\p)[2] \simeq \Zt \times \Zt$ and $E^\prime(K_\p)[2] \simeq \Zt$, then $\dimF H^1_\phi(K_\p, C^\chi) = 0$.
\end{enumerate}
\end{lemma}
\begin{proof}
From Lemma 3.7 in \cite{K}, we have 
\[
E^{\prime \chi}(K_\p)[2^\infty]/\phi( E^\chi(K_\p)[2^\infty]) = E^{\prime \chi}(K_\p)[2]/\phi( E^\chi(K_\p)[2]).
\]
All four results then follow immediately.
\end{proof}

\begin{remark}\label{remleg}
Suppose that $\Delta$ and $\Delta^\prime$ are discriminants of any integral models of $E$ and $E^\prime$ respectively. The condition that $E(K_\p)[2] \simeq \Zt$ ( resp. $E^\prime(K_\p)[2]\simeq \Zt$) is equivalent to the condition that $\Delta$ (resp. $\Delta^\prime$) is not a square in $K_\p$. Which case of Lemma \ref{localconds} we are is therefore determined by the Legendre symbols $\left ( \frac{\Delta}{\p} \right )$ and $\left(  \frac{\Delta^\prime}{\p}  \right )$.

\end{remark}
We use Theorem \ref{prodform2} and Lemma \ref{localconds} to relate $\ord_2 \T(E^\chi/E^{\prime \chi})$ to the value $g(\chi)$ of an additive function $g$ on $C(K)$ defined  as follows: For $\chi \in C(K)$ cutting out $F/K$, let \begin{equation}\label{gdef} g(\chi) = \sum_{\genfrac{}{}{0pt}{}{\p \text{ ramified in } F/K}{\p \nmid 2\Delta \infty}} \frac{ \left ( \frac{\Delta^\prime}{\p} \right ) - \left(  \frac{\Delta}{\p}  \right ) }{2}\end{equation} That is, $g(\chi)$ roughly counts the difference between the number of primes ramified in $F/K$ where condition $(iii)$ of Proposition \ref{localconds} is satisfied and the number of primes ramified in $F/K$ where condition $(iv)$ is satisfied. We then have the following:




\begin{proposition}\label{TT}
The order of $2$ in the Tamagawa ratio $\T(E^\chi/E^{\prime \chi})$ is given by \begin{equation*}\ord_2 \T(E^\chi/E^{\prime \chi}) = g(\chi) + \sum_{v |2\Delta\infty} \left ( \dimF H^1_\phi(K_v, C^\chi)  - 1 \right ).\end{equation*}
\end{proposition}
\begin{proof}
By Thereom \ref{prodform2}, $\ord_2 \T(E^\chi/E^{\prime \chi})$ is given by \begin{equation*}\displaystyle{ \ord_2 \T(E^\chi/E^{\prime \chi}) = \sum_{v |2\Delta\Delta_{F/K}\infty} \left ( \dimF H^1_\phi(K_v, C^d)  - 1 \right ),}\end{equation*} where $\Delta_{F/K}$ is the relative discriminant of the extension $F/K$ cut out by $\chi$. By Remark \ref{remleg}, Lemma \ref{localconds} gives us that \begin{equation*}\dimF H^1_\phi(K_\p, C^\chi) - 1 =  \frac{ \left ( \frac{\Delta^\prime}{\p} \right ) - \left(  \frac{\Delta}{\p}  \right ) }{2}\end{equation*} for places $\p \mid \Delta_{F/K}$ with $\p\nmid 2\Delta\infty$ and the result follows.
\end{proof}

\section{The Erd{\H o}s-Kac Theorem For Quadratic Characters}\label{pfEKthm}

Because the sum $$ \sum_{v |2\Delta\infty} \left ( \dimF H^1_\phi(K_v, C^\chi)  - 1 \right )$$ can be bounded uniformly for a fixed elliptic curve, Proposition \ref{TT} suggests that we should study the distribution of the additive function $g(\chi)$ in order to understand the distribution of $\T(E^\chi/E^{\prime \chi})$ as $\chi$ varies.

When $f:\mathbb{N} \rightarrow \C$ is an additive function on the integers, then under mild hypotheses, the classical Erd{\H o}s-Kac Theorem tells us that the distribution of $f$ on natural numbers less than $X$ approaches a normal distribution with mean $\mu(X)$ and variance $\sigma^2(X)$ as $X \rightarrow \infty$, where $\mu(X)$ and $\sigma(X)$ are the rational analogues of $\mu_f(X)$ and $\sigma_f(X)$ defined in the introduction. Theorem \ref{EKthm} is the statement that the same type of result holds for additive functions on quadratic characters of a number field.

We begin by observing that if $\chi_{d_1}=\chi_{d_2}$, then, if $(d_1)=\mathfrak{a}\mathfrak{b}_1^2$ and $(d_2)=\mathfrak{a}\mathfrak{b}_2^2$, $\mathfrak{b}_1$ and $\mathfrak{b}_2$ lie in the same ideal class of $\mathcal{O}_K$.  Conversely, if $\mathfrak{b}_1$ and $\mathfrak{b}_2$ lie in the same class, then $\chi_{d_1}=\chi_{\varepsilon d_2}$ for some $\varepsilon\in\mathcal{O}_K^\times$.  Thus, we see that elements of $C(K, X)$ correspond to triples $(\mathfrak{b},\mathfrak{a},\varepsilon)$ with $\mathfrak{b}$ a representative of its ideal class of minimal norm, $\mathfrak{a}$ a squarefree ideal of norm $\N\mathfrak{a} <X/\N\mathfrak{b}^2$ such that $\mathfrak{a}\mathfrak{b}^2$ is principal, and $\varepsilon$ an element of $\mathcal{O}_K^\times/(\mathcal{O}_K^\times)^2$.  From this discussion, it is apparent, for a prime ideal $\mathfrak{p}$ and a fixed choice $\mathfrak{b}_0$ of $\mathfrak{b}$, that
\begin{eqnarray*}
\lim_{X\to\infty} \frac{|\left \{\chi\in C(K,X) \leftrightarrow (\mathfrak{b}_0,\mathfrak{a},\varepsilon) : \mathfrak{p}\!\mid\!\mathfrak{a} \right \} |}{|C(K,X)|} 
	&=& \mathrm{Pr}(\mathfrak{p}\!\mid\!\mathfrak{a} : \mathfrak{a} \text{ is squarefree}, \mathfrak{a}\mathfrak{b}_0^2 \text{ is principal}) \\
	&=& \frac{1}{\N\mathfrak{p}+1},
\end{eqnarray*}
whence, ranging over all $\chi$, the probablity that $\mathfrak{p}\mid\mathfrak{a}$ is $\frac{1}{\N\mathfrak{p}+1}$.

Thus, treating the events $\mathfrak{p}_1\mid\mathfrak{a}$ and $\mathfrak{p}_2\mid\mathfrak{a}$ as independent, we might predict that 
\[
\frac{1}{|C(K,X)|} \sum_{\chi\in C(K,X)} f(\chi) \approx \tilde{\mu}_f(X),
\]
where
\begin{eqnarray*}
\tilde{\mu}_f(X) 
	&:=& \sum_{\N\mathfrak{p} < X} \frac{f(\mathfrak{p})}{\N\mathfrak{p}+1} \\
	&=& \mu_f(X) + O(1).
\end{eqnarray*}
In fact, this prediction is true, which we establish using the method of moments, following the blueprint of Granville and Soundararajan \cite{GS}.  This requires the following technical result about the function $g_\mathfrak{p}(\chi)$, defined to be
\[
g_\mathfrak{p}(\chi) := \left\{ \begin{array}{ll} 
f(\mathfrak{p})\left(1-\frac{1}{\N\mathfrak{p}+1}\right) & \text{if } \mathfrak{p}\mid\mathfrak{a}, \text{ and} \\ 
f(\mathfrak{p})\left(- \frac{1}{\N\mathfrak{p}+1}\right) & \text{if } \mathfrak{p}\nmid\mathfrak{a}.
\end{array}\right.
\]
\begin{theorem}
\label{thm:additive}
With notation as above, uniformly for $k\leq \sigma_f(z)^{2/3}$, we have that
\[
\frac{1}{|C(K,X)|}\sum_{\chi\in C(K,X)} \left(\sum_{\N\mathfrak{p}<z} g_\mathfrak{p}(\chi)\right)^k = c_k \sigma_f(z)^k\left(1+O\left(\frac{k^3}{\sigma_f(z)^2}\right)\right)+O\left(X^{\lambda-1} 3^k \pi_K(z)^k\right)
\]
if $k$ is even, and
\[
\frac{1}{|C(K,X)|}\sum_{\chi\in C(K,X)} \left(\sum_{\N\mathfrak{p}<z} g_\mathfrak{p}(\chi)\right)^k \ll c_k \sigma_f(z)^{k-1} k^{3/2} + X^{\lambda-1} 3^k \pi_K(z)^k
\]
if $k$ is odd.  Here, $c_k=\Gamma(k+1)/2^{k/2}\Gamma(\frac{k}{2}+1)$ and $\lambda<1$ depends only on the degree of $K$.
\end{theorem}

The proof of Theorem \ref{thm:additive} relies upon a result about the distribution of squarefree ideals.  To this end, given ideals $\mathfrak{c}$ and $\mathfrak{q}$ and squarefree $\mathfrak{d}\mid\mathfrak{q}$, define $N^{\mathrm{sf}}(X;\mathfrak{c},\mathfrak{q},\mathfrak{d})$ to be the number of squarefree ideals $\mathfrak{a}$ of norm up to $X$ in the same class as $\mathfrak{c}$ and such that $(\mathfrak{a},\mathfrak{q})=\mathfrak{d}$.  We then have:

\begin{lemma}
\label{lem:dist}
With notation as above, we have that
\[
N^{\mathrm{sf}}(X;\mathfrak{c},\mathfrak{q},\mathfrak{d}) = \frac{1}{|\mathrm{Cl}(K)|}\frac{\mathrm{res}_{s=1}\zeta_K(s)}{\zeta_K(2)} \phi(\mathfrak{q},\mathfrak{d}) X + O(X^\lambda 3^{\omega(\mathfrak{q})}),
\]
where $\lambda=\frac{\deg K-1}{\deg K+1}$ if $\deg K \geq 3$ and $\lambda=1/2$ otherwise, $\omega(\mathfrak{q})$ denotes the number of distinct primes dividing $\mathfrak{q}$, and
\[
\phi(\mathfrak{q},\mathfrak{d}) = \prod_{\mathfrak{p}\mid\mathfrak{d}} \frac{1}{\N\mathfrak{p}+1} \prod_{\mathfrak{p}\mid\mathfrak{q},\mathfrak{p}\nmid\mathfrak{d}} \frac{\N\mathfrak{p}}{\N\mathfrak{p}+1}.
\]
\end{lemma}
\begin{proof}
This follows from elementary considerations and the classical estimate
\[
\sum_{\begin{subarray}{c} \N\mathfrak{a} < X \\ \mathfrak{a}\mathfrak{c}^{-1} \text{ prin.} \end{subarray}} 1 = \frac{1}{|\mathrm{Cl}(K)|}\mathrm{res}_{s=1}\zeta_K(s)\cdot X + O(X^{\frac{\deg K-1}{\deg K+1}}).
\]
\end{proof}

\begin{proof}[Proof of Theorem \ref{thm:additive}]
For any ideal $\mathfrak{q}$, define $g_\mathfrak{q}(\chi):=\prod_{\mathfrak{p}^\alpha \mid\mid \mathfrak{q}}g_\mathfrak{p}(\chi)^\alpha$.  We then have that
\begin{eqnarray*}
\sum_{\chi\in C(K,X)} \left(\sum_{\N\mathfrak{p}<z} g_\mathfrak{p}(\chi)\right)^k
	&=& \left|\mathcal{O}_K^\times/(\mathcal{O}_K^\times)^2\right| \cdot \sum_{\mathfrak{b}} \sideset{}{^\prime}\sum_{\begin{subarray}{c}\N\mathfrak{a}<\frac{X}{\N\mathfrak{b}^2} \\ \mathfrak{a}\mathfrak{b}^2\, \text{prin.} \end{subarray}} \left(\sum_{\N\mathfrak{p}<z} g_\mathfrak{p}(\mathfrak{a})\right)^k \\
	&=& \left|\mathcal{O}_K^\times/(\mathcal{O}_K^\times)^2\right| \cdot \sum_{\mathfrak{b}} \sum_{\N\mathfrak{p}_1,\dots,\N\mathfrak{p}_k < z} \sideset{}{^\prime}\sum_{\begin{subarray}{c}\N\mathfrak{a}<\frac{X}{\N\mathfrak{b}^2} \\ \mathfrak{a}\mathfrak{b}^2\, \text{principal} \end{subarray}} g_{\mathfrak{p}_1\dots\mathfrak{p}_k}(\mathfrak{a}),
\end{eqnarray*}
where, as expected, the summation over $\mathfrak{b}$ is taken to be over representatives of minimal norm for each ideal class and the prime on the summation over $\mathfrak{a}$ indicates it is to be taken over squarefree ideals.  Given an ideal $\mathfrak{c}$, we now consider for any $\mathfrak{q}$ the more general summation
\begin{eqnarray*}
\sideset{}{^\prime}\sum_{\begin{subarray}{c} \N\mathfrak{a}<Y \\ \mathfrak{a}\mathfrak{c}^{-1} \,\text{principal} \end{subarray}} g_\mathfrak{q}(\mathfrak{a}) 
	&=& \sum_{\mathfrak{d}\mid \sqrt{\mathfrak{q}}} g_\mathfrak{q}(\mathfrak{d}) N^{\mathrm{sf}}(Y;\mathfrak{c},\mathfrak{q},\mathfrak{d}) \\
	&=& \frac{Y}{|\mathrm{Cl}(K)|} \frac{\mathrm{res}_{s=1}\zeta_K(s)}{\zeta_K(2)} \sum_{\mathfrak{d}\mid \sqrt{\mathfrak{q}}} g_\mathfrak{q}(\mathfrak{d})\phi(\mathfrak{q},\mathfrak{d}) + O\left(Y^{\lambda}3^{\omega(\mathfrak{q})}\sum_{\mathfrak{d}\mid \sqrt{\mathfrak{q}}} |g_\mathfrak{q}(\mathfrak{d})| \right), \\
	&=:& \frac{Y}{|\mathrm{Cl}(K)|} \frac{\mathrm{res}_{s=1}\zeta_K(s)}{\zeta_K(2)} G(\mathfrak{q}) + O\left(Y^\lambda 3^{\omega(\mathfrak{q})}\right),
\end{eqnarray*}
say, where $\sqrt{\mathfrak{q}}=\prod_{\mathfrak{p}\mid\mathfrak{q}}\mathfrak{p}$.  We note that $G(\mathfrak{q})$ is multiplicative, and is given by
\[
G(\mathfrak{q}) = \prod_{\mathfrak{p}^\alpha\mid\mid \mathfrak{q}} \frac{f(\mathfrak{p})^\alpha}{\N\mathfrak{p}+1} \left( \left(1-\frac{1}{\N\mathfrak{p}+1}\right)^\alpha + \N\mathfrak{p} \cdot \left(-\frac{1}{\N\mathfrak{p}+1}\right)^\alpha \right).
\]
Thus, $G(\mathfrak{q})=0$ unless each $\alpha$ is at least two, i.e. $\mathfrak{q}$ is square-full.

Returning to the original problem, we find that
\begin{eqnarray*}
\sum_{\chi\in C(K,X)} \left(\sum_{\N\mathfrak{p}<z} g_\mathfrak{p}(\chi)\right)^k
	&=& c(K)\cdot X \sum_{\N\mathfrak{p}_1,\dots,\N\mathfrak{p}_k<z}G(\mathfrak{p}_1\dots\mathfrak{p}_k) + O\left(X^{\lambda}3^k\pi_K(z)^k\right),
\end{eqnarray*}
where $\pi_K(z):=\#\{\mathfrak{p} : \N\mathfrak{p}<z\}$ and
\[
c(K):=\left| \mathcal{O}_K^\times/(\mathcal{O}_K^\times)^2 \right| \frac{1}{|\mathrm{Cl}(K)|} \frac{\mathrm{res}_{s=1}\zeta_K(s)}{\zeta_K(2)} \sum_{\mathfrak{b}} \frac{1}{\N\mathfrak{b}^2}.
\]
Noting that the above discussion also proves that $|C(K, X)| = c(K)\cdot X + O(X^{\lambda})$, the goal is to estimate the summation over $\mathfrak{p}_1,\dots,\mathfrak{p}_k$.  Since $G(\mathfrak{q})=0$ unless $\mathfrak{q}$ is square-full, we have that
\[
\sum_{\N\mathfrak{p}_1,\dots,\N\mathfrak{p}_k<z}G(\mathfrak{p}_1\dots\mathfrak{p}_k) = \sum_{s\leq k/2} \sum_{\begin{subarray}{c} \alpha_1+\dots+\alpha_s=k, \\ \text{each }\alpha_i\geq 2 \end{subarray}} \frac{k!}{\alpha_1!\dots\alpha_s!}  \sum_{\begin{subarray}{c} \mathfrak{p}_1 < \dots < \mathfrak{p}_s \\ \N\mathfrak{p}_s < z \end{subarray}} G(\mathfrak{p}_1^{\alpha_1}\dots\mathfrak{p}_s^{\alpha_s}), 
\]
where $\mathfrak{p}<\mathfrak{p}^\prime$ is determined by a norm-compatible linear ordering on the prime ideals of $\mathcal{O}_K$.  We note that, since $G(\mathfrak{p}^\alpha) \leq \frac{f(\mathfrak{p})^2}{N\mathfrak{p}}$, the inner summation contributes no more than $O(\sigma_f(z)^{2s})$, which will be an error term unless $s=\frac{k}{2}$; the dependence of this error on $k$ can be sussed out exactly as in Granville and Soundararajan's work.  In fact, the handling of the main term, arising when $k$ is even and $s=\frac{k}{2}$, is also essentially the same.  In particular, the inner summation is equal to
\begin{eqnarray*}
\frac{1}{(\frac{k}{2})!} \sum_{\begin{subarray}{c} \N\mathfrak{p}_1,\dots,\N\mathfrak{p}_{\frac{k}{2}}<z \\ \mathrm{distinct} \end{subarray}} G(\mathfrak{p}_1^2\dots\mathfrak{p}_{\frac{k}{2}}^2) 
	&=& \frac{1}{(\frac{k}{2})!} \left(\sum_{\N\mathfrak{p}<z} G(\mathfrak{p}^2) + O(\log\log k) \right)^{k/2} \\
	&=& \frac{1}{(\frac{k}{2})!} \left(\sigma_f(z)^2+O(\log\log k)\right)^{k/2}.
\end{eqnarray*}
This yields Theorem \ref{thm:additive}.
\end{proof}

We are now ready to prove Theorem \ref{EKthm}.

\begin{proof}[Proof of Theorem \ref{EKthm}]
Recall that we wish to show that the quantity
\[
\frac{f(\chi)-\mu_f(X)}{\sigma_f(X)}, \,\, \chi\in C(K,X),
\]
is normally distributed as $X\to\infty$.  As remarked above, we will do so using the method of moments.  In particular, we have that
\begin{eqnarray*}
\frac{1}{|C(K,X)|} \sum_{\chi\in C(K,X)}\!\!\!\!\!\! \left(f(\chi)-\mu_f(X)\right)^k\!\!\!\!
	&=& \!\!\!\!\frac{1}{|C(K,X)|} \sum_{\chi\in C(K,X)}\!\! \left(\sum_{\mathfrak{p}\mid\mathfrak{a}}f(\mathfrak{p})-\sum_{\N\mathfrak{p}<X} \frac{f(\mathfrak{p})}{\N\mathfrak{p}+1} + O(1) \right)^k \\
	&=& \!\!\!\! \frac{1}{|C(K,X)|} \sum_{\chi\in C(K,X)}\!\! \left(\sum_{\N\mathfrak{p}<z} g_\mathfrak{p}(\mathfrak{a}) + O\left(\frac{\log X}{\log z} \right) \right)^k.
\end{eqnarray*}
Considering the error term in Theorem \ref{thm:additive}, we take $z=X^{\frac{1-\lambda}{k}}$.  With this choice, the inner summation becomes
\[\left(\sum_{\N\mathfrak{p}<z} g_\mathfrak{p}(\mathfrak{a}) \right)^k+O\left(\sum_{j=0}^{k-1} k^j \left|\sum_{\N\mathfrak{p}<z} g_\mathfrak{p}(\mathfrak{a}) \right|^{k-j}\right),\]
hence Theorem \ref{thm:additive}, the Cauchy-Schwarz inequality, and the fact that $\sigma_f(z)=\sigma_f(X)+O(\log k)$ yield that
\[
\frac{1}{|C(K,X)|} \sum_{\chi\in C(K,X)} \left(f(\chi)-\mu_f(X)\right)^k = c_k \sigma_f(X)^k\left(1 + O\left( \frac{k^{3/2}}{\sigma_f(X)}\right)\right)
\]
if $k$ is even, and
\[
\frac{1}{|C(K,X)|} \sum_{\chi\in C(K,X)} \left(f(\chi)-\mu_f(X)\right)^k \ll c_k \sigma_f(X)^{k-1} k^{3/2}
\]
if $k$ is odd.  This proves the theorem.
\end{proof}

\section{Proof of Main Theorems}\label{pfofmain}
In order to apply Theorem \ref{EKthm} to the additive function $g(\chi)$ defined in (\ref{gdef}), we need to evaluate the quantities $\mu_g(X)$ and $\sigma_g(X)$. We begin with the following




\begin{proposition}\label{sumprop}
Let $c \in K^\times$ be non-square. Then \begin{equation*}\sum_{\N\p \le X}\frac{1 +\left (\frac{c}{\p} \right ) }{\mathbf{N}\p} = \log \log X + O(1).\end{equation*}
\end{proposition}
\begin{proof}
This is a consequence of the prime ideal theorem and the fact that the Hecke $L$-function attached to the non-trivial character of $\mathrm{Gal}(K(\sqrt{c})/K)$ is analytic and non-vanishing on the line $\Re(s)=1$.
\end{proof}

We now decompose $\mu_g(X)$ as \begin{equation*}\mu_g(X) = \frac{1}{2}\sum_{\genfrac{}{}{0pt}{}{\mathbf{N}\p < X}{p \nmid 2\Delta \infty}}\frac{1 +\left (\frac{\Delta^\prime}{\p} \right ) }{\mathbf{N}\p} - \frac{1}{2}\sum_{\genfrac{}{}{0pt}{}{\mathbf{N}\p < X}{\p \nmid 2\Delta \infty}}\frac{1 +\left (\frac{\Delta}{\p} \right ) }{\mathbf{N}\p},\end{equation*} and it immediately follows from Proposition \ref{sumprop} that $\mu_g(X) = O(1)$. We also rewrite $\sigma_g(X)$ as \begin{equation*}\sigma_g(X) =\left (  \displaystyle{ \sum_{\genfrac{}{}{0pt}{}{\mathbf{N} \p \le X,  \p \nmid 2\Delta\infty }{ \left (\frac{\Delta}{\p} \right ) \ne \left ( \frac{\Delta^\prime}{\p} \right )} }\frac{1}{\mathbf{N} \p}} \right )^{1/2} =\left (\frac{1}{2}\displaystyle{ \sum_{\genfrac{}{}{0pt}{}{\mathbf{N}\p \le X}{\p \nmid 2\Delta\infty }}\frac{1 - \left (\frac{\Delta\Delta^\prime}{\p} \right ) }{\mathbf{N} \p}} \right )^{1/2}.\end{equation*}

In order to apply Proposition \ref{sumprop} to $\sigma_g(X)$, we therefore need $\Delta\Delta^\prime$ to be non-square in $K$. This will be the case when $E$ does not have a cyclic 4-isogeny defined over $K(E[2])$.

\begin{proposition}\label{charac}
If $E$ is an elliptic curve with $E(K)[2] \simeq \Zt$ that does not have a cyclic 4-isogeny defined over $K(E[2])$, then $\Delta\Delta^\prime \not \in (K^\times)^2$
\end{proposition}
\begin{proof}
Let $Q^\prime \in E^\prime[2] - C^\prime$, $C = \langle P \rangle$, and take $Q \in E[4]$ with $\phi(Q) = Q^\prime$. Since $Q^\prime \in E^\prime(K)[2] - C^\prime$, and both $\phi \circ \hatphi = [2]_{E^\prime}$ and $\hatphi \circ \phi = [2]_{E}$, it follows that $2Q = \hatphi(Q^\prime) = P.$ Let $M = K(E[2])$. Since $E$ has no cyclic 4-isogeny defined over $M$, there exists $\sigma \in G_M$ such that $\sigma(Q) \not \in \langle Q \rangle = \left \{0 , Q, P, Q + P \right \}.$ In particular, since $\phi^{-1}(Q^\prime) \subset \langle Q \rangle$, we get that $\phi(\sigma(Q)) \ne Q^\prime$. We then get that \begin{equation*}\sigma(Q^\prime) = \sigma(\phi(Q)) = \phi(\sigma(Q)) \ne Q^\prime,\end{equation*} showing that $Q^\prime$ is not defined over $M$, and therefore that $K(E^\prime[2]) \not \subset M$. It then follows that $K(E[2])$ and $K(E^\prime[2])$ are disjoint quadratic extensions of $K$ and that $E^\prime(K)[2] \simeq \Zt$. As $K(E[2])$ and $K(E^\prime[2])$ are given by $K(\sqrt{\Delta})$ and $K(\sqrt{\Delta^\prime})$ respectively, it follows that $\Delta\Delta^\prime \not \in (K^\times)^2$.
\end{proof}


By Proposition \ref{sumprop}, we therefore get that $\sigma_g(X) = \sqrt{\frac{1}{2}\log \log X}+ O(1)$ whenever $E$ does not have a cyclic 4-isogeny defined over $K(E[2])$.


\begin{proof}[Proof of Theorem \ref{Tconverge}]
Applying Theorem \ref{EKthm} to $g(\chi)$, we get that \begin{equation}\label{limeq} \lim_{X\to\infty} \frac{ \left | \left \{ \chi\in C(K, X) : g(\chi) - O(1) \leq z \left (\sqrt{\frac{1}{2} \log \log X} + O(1) \right ) \right \} \right |}{ |C(K,X)|} = G(z) \end{equation} for every $z \in \R$. By Proposition \ref{TT}, there is some contant $C$, independent of $\chi$, such that $|g(\chi) - \ord_2 \T(E^\chi/E^{\prime \chi})| <  C$, so in fact (\ref{limeq}) holds with $g(\chi)$ replaced by $\ord_2 \T(E^\chi/E^{\prime \chi})$ and the result follows.
\end{proof}

\begin{proof}[Proof of Theorem \ref{mainthm}]
By Theorem \ref{Tconverge}, \begin{equation*}\lim_{X \rightarrow \infty} \displaystyle{ \frac{ |\{\chi \in C(K,X) : \ord_2 \T(E^\chi/E^{\prime \chi}) \ge r \}| }{|C(K, X)|} } = \frac{1}{2}\end{equation*} for any fixed $r \ge 0$. As $d_2(E^\chi/K) \ge  \ord_2 \T(E^\chi/E^{\prime \chi}) - 2$, this shows that for any $\epsilon > 0$, \begin{equation*}\displaystyle{ \frac{ |\{\chi \in C(K, X) : d_2 (E^\chi/K) \ge r \}| }{|C(K, X)|} } \ge \frac{1}{2} - \epsilon\end{equation*} for sufficiently large $X$.
\end{proof}

\bibliography{citations.bib}

\end{document}